\documentclass{amsart}
\usepackage{braket}

\newtheorem{thm}{Theorem}
\newtheorem{prethm}{Theorem}

\newtheorem*{conj}{Conjecture}
\newtheorem{lem}{Lemma}

\begin{document}

\title[The Graph Ramsey Number $R(F_\ell,K_6)$]
{The Graph Ramsey Number $R(F_\ell,K_6)$}

\author[S. Kadota]{Shin-ya Kadota}
\author[T. Onozuka]{Tomokazu Onozuka}
\author[Y. Suzuki]{Yuta Suzuki}



\maketitle


\begin{abstract}
For a given pair of two graphs $(F,H)$,
let $R(F,H)$ be the smallest positive integer $r$
such that for any graph $G$ of order $r$,
either $G$ contains $F$ as a subgraph
or
the complement of $G$ contains $H$ as a subgraph.
Baskoro, Broersma and Surahmat (2005) conjectured that
\[
R(F_\ell,K_n)=2\ell(n-1)+1
\]
for $\ell\ge n\ge3$, where $F_\ell$ is the join of $K_1$ and $\ell K_2$.
In this paper, we prove that this conjecture is true for the case $n=6$.
\end{abstract}

\section{Introduction}
Throughout this paper, all graphs are finite and simple.
For a given pair of two graphs~$(F,H)$,
let $R(F,H)$ be the smallest positive integer $r$
such that for any graph $G$ of order $r$,
$G$ contains $F$ as a subgraph
or
the complement of $G$ contains $H$ as a subgraph.
In general, it is quite difficult to calculate the exact values of $R(F,H)$.
However, for sparse graphs $F$ and $H$,
there are many results on the exact values of $R(F,H)$.
In this paper, we consider fan graph $F_\ell$,
which is the join of $K_1$ and $\ell K_2$.
This fan graph is one example of such sparse graphs.

Recently,
although $K_n$ itself is the densest graph,
some authors succeeded in determining the value $R(F_\ell,K_n)$
for large $\ell$ and small $n$.
As for this problem,
Baskoro, Broersma and Surahmat~%
\cite{Baskoro_Broersma_Surahmat} conjectured:
\begin{conj}[The Baskoro-Broersma-Surahmat conjecture~%
\cite{Baskoro_Broersma_Surahmat}]
\mbox{}

\begin{center}
For any $\ell\ge n\ge3$, we have $R(F_\ell,K_n)=2\ell(n-1)+1$.
\end{center}
\end{conj}
In the last 20 years,
this conjecture was proved to be true for the case $n=3,4,5$:
\begin{prethm}[Gupta, Gupta and Sudan~\cite{Gupta_Gupta_Sudan},
{Li and Rousseau~\cite[Proposition 1]{Li_Rousseau}}]
\label{Gupta_Gupta_Sudan}
\mbox{}

\begin{center}
For any $\ell\ge3$, we have $R(F_\ell,K_3)=4\ell+1$.
\end{center}
\end{prethm}
\begin{prethm}[Baskoro, Broersma and Surahmat~%
\cite{Baskoro_Broersma_Surahmat}]
\label{Baskoro_Broersma_Surahmat}
\mbox{}

\begin{center}
For any $\ell\ge4$, we have $R(F_\ell,K_4)=6\ell+1$.
\end{center}
\end{prethm}
\begin{prethm}[Chen and Zhang~%
\cite{Chen_Zhang}]
\label{Chen_Zhang}
\mbox{}

\begin{center}
For any $\ell\ge5$, we have $R(F_\ell,K_5)=8\ell+1$.
\end{center}
\end{prethm}

In this paper,
we prove the Baskoro-Broersma-Surahmat conjecture for $n=6$.
\begin{thm}
\label{main_thm}
\mbox{}

\begin{center}
For any $\ell\ge6$, we have $R(F_\ell,K_6)=10\ell+1$.
\end{center}
\end{thm}
Our main strategy is
based on the approach of Chen and Zhang~\cite{Chen_Zhang}.
In particular, we shall heavily use their structural setting.
However, we need some additional insight on the Chen-Zhang structure
than their paper~\cite{Chen_Zhang}.

We briefly summarize our notation and terminology.
Let $G=(V,E)$ be a graph.
For any subset $S\subset V$, we use $G[S]$ to denote the subgraph induced by $S$.
For any vertex $v\in V$,
we denote the neighborhood of $v$ by $N(v)$,
i.e.
\begin{align*}
N(v)=&
\Set{w\in V|vw\in E,\ w\neq v},
\end{align*}
and we let $N[v]=N(v)\cup\{v\}$.
We let
\[
d(v)=|N(v)|,\quad
\delta(G)=\min_{v\in V}d(v),\quad
\Delta(G)=\max_{v\in V}d(v).
\]
Moreover, we denote by $\omega(G)$ the clique number of $G$,
i.e.\ the order of the largest clique in $G$,
and denote by $\alpha(G)$ the independence number of $G$,
i.e.\ the order of the largest independent set in $G$.
We refer to the book~\cite{Diestel}
for other graph theoretical notation and terminology
not described in this paper.

\section{The lower bound}
The lower bound of $R(F_\ell,K_6)$ is given by the following theorem,
which is just a special case of the Chv\'atal-Harary lemma~%
\cite[Lemma~4]{Chvatal_Harary}.
\begin{thm}
\label{lower_bound}
For any $\ell\ge 6$,
we have
$R(F_\ell,K_6)\ge10\ell+1$.
\end{thm}
\begin{proof}
It is sufficient to give a graph $G$ of order $10\ell$
such that $G$ does not contain $F_\ell$
and $\overline{G}$ does not contain $K_6$.
For example, the graph
$G=5K_{2\ell}$
satisfies this condition.
Hence the theorem follows.
\end{proof}
Hence we prove the upper bound $R(F_\ell,K_6)\le10\ell+1$ in the following sections.
\begin{thm}
\label{upper_bound}
For any $\ell\ge 6$,
we have
$R(F_\ell,K_6)\le10\ell+1$.
\end{thm}
We prove this upper bound by contradiction.
In the remaining part of this paper,
we assume $\ell\ge6$ and that there exists a graph $G$ of order $10\ell+1$
such that $F_\ell\not\subset G$ and $K_6\not\subset\overline{G}$.
By Theorem~\ref{Chen_Zhang},
we may assume that $\alpha(G)=5$.

\section{Preliminary lemmas}
In this section,
we prove several lemmas on basic properties of the graph $G$.
We start with the following simple observation,
which is related to the number of independent edges in $G$.

\begin{lem}
There is no vertex $v\in G$ for which $G[N(v)]$ contains $\ell K_2$.
\end{lem}
\begin{proof}
If $\ell K_2\subset G[N(v)]$,
then $F_\ell\subset G[N[v]]$. This is a contradiction.
\end{proof}

The next lemma is a special case of Stahl's lemma~\cite{Stahl}.
\begin{lem}
\label{matching}
For any $\ell\ge1$, we have $R(\ell K_2,K_6)=2\ell+4$.
\end{lem}
\begin{proof}
See~\cite[pp.~586--587]{Stahl}.
\end{proof}

As a consequence of Lemma~\ref{matching} and $\alpha(G)=5$,
any subgraph $H\subset G$ of order $n\ge6$
contains $\lfloor n/2\rfloor-2$ or more independent edges.
The following is an immediate consequence
of Theorem~\ref{Chen_Zhang} and Lemma~\ref{matching}.

\begin{lem}
\label{degree_lemma}
We have $2\ell\le\delta(G)\le\Delta(G)\le2\ell+3$.
\end{lem}

\begin{proof}
First we show that $\delta(G)\ge 2\ell$. 
Assume, to the contrary, that there exists a vertex $v$ with $d(v)\le 2\ell-1$. 
Let $S= V(G)\backslash N[v]$. 
Then $|S|\ge 8\ell+1$. 
Also, since $\alpha(G)= 5$, it follows that $\alpha(G[S])\le 4$. 
Then $G[S]$ contains $F_{\ell}$ by Theorem~\ref{Chen_Zhang}, a contradiction.
To show that $\Delta(G)\le 2\ell+3$, suppose that $d(v)\ge 2\ell+4$ for some vertex $v$. 
Since $\alpha(G)= 5$, Lemma~\ref{matching} guarantees the existence of $\ell K_2$ in $G[N(v)]$, producing $F_{\ell}$ in $G$. 
This cannot occur and so $\Delta(G)\le 2\ell+3$, as claimed.
\end{proof}

The next lemma estimates the clique number $\omega(G)$.
\begin{lem}
\label{large_clique}
We have $\omega(G)\le2\ell-2$.
\end{lem}
\begin{proof}
Assume, to the contrary, that $G$ contains a clique $H$ of order $2\ell-1$. 
Select a vertex $v_0\in V(G)\backslash V(H)$ such that 
\[
|N(v_0)\cap V(H)|
=
\max\{\,|N(v)\cap V(H)|\mid v\in V(G)\backslash V(H)\,\}.
\]
The graph $G-H-v_0$ is of order $8\ell+1$
so that $\alpha(G-H-v_0)=5$ by Theorem~\ref{Chen_Zhang}.
Let $U$ be a $5$-set of independent vertices in the graph $G_0= G-H-v_0$. 

Since $U\cup \{v\}$ cannot be independent for each $v\in V(G)\backslash U$, there are at least $2\ell-1$ edges between $H$ and $U$. 
Hence, there exists a vertex $u_0\in U$ with $|N(u_0)\cap V(H)|\ge (2\ell-1)/5> 2$. 
Consequently, $|N(v_0)\cap V(H)|\ge |N(u_0)\cap V(H)|\ge 3$. 

Next we show that, for each vertex $v\in V(G_0)$, if $w\in N(v_0)\cap N(v)\cap V(H)$, then $N(v)\cap V(H)= \{w\}$. 
If this is not the case, say there exists a vertex $w'\in N(v)\cap V(H)$ with $w\ne w'$, then $(\ell-1)K_2$ in $G[(V(H)\backslash \{w,w'\})\cup \{v_0\}]$ and the edge $vw'$ form $\ell K_2$ in $G[N(w)]$, which is impossible. 

Now, let $S_1= N(v_0)\cap V(H)= \{w_1,w_2,\ldots,w_t\}$, where $t= |N(v_0)\cap V(H)|$.
Then by the above observations,
we can find a $t$-set $U_1= \{u_1,u_2,\ldots,u_t\}\subset U$
such that $N(u_i)\cap V(H)= \{w_i\}$ for $1\le i\le t$. 
Also, let $S_2=V(H)\setminus S_1$ and $U_2= U\backslash U_1$. 
Recall that $|N(u_0)\cap V(H)|\ge 3$, so $u_0\in U_2$. 
Hence, $t= 3,4$. 

Note that there is no edge between $U_1$ and $S_2$
and also note that $U\cup\{v\}$ cannot be independent for each $v\in S_2$.
Thus there are at least $|S_2|=2\ell-1-t$ edges between $U_2$ and $S_2$.
However then, 
\[
\frac{11-t}{5-t}
\le
\frac{2\ell-1-t}{|U_2|}
\le
\max_{u\in U_2}|N(u)\cap V(H)|
\le
|N(v_0)\cap V(H)|= t,
\]
which cannot occur. 
This completes the proof.
\end{proof}

\section{Structural observation}
In this section,
we give some observation on the structure of the graph $G$
following the argument of Chen and Zhang~\cite{Chen_Zhang}.
Let
\[U=\{u_1,u_2,u_3,u_4,u_5\}\subset V(G)\]
be a 5-set of independent vertices and let
\[
X_i:=\Set{v\in V(G)|{|N(v)\cap U|}=i}
\]
for $1\le i\le 5$.
Since $\alpha(G)\le 5$,
we have a partition
\[
V(G)=U\sqcup X_1\sqcup X_2\sqcup X_3\sqcup X_4\sqcup X_5.
\]

\subsection{On the sets $X_i$}
Obviously, we have
\begin{equation}
\label{X_total}
\sum_{i=1}^5|X_i|=10\ell-4
\end{equation}
and, since $d(u_i)\le\Delta(G)\le2\ell+3$ by Lemma~\ref{degree_lemma}, we have
\begin{equation}
\label{bridge_whole_X}
\sum_{i=1}^{5}i|X_i|
=\sum_{i=1}^{5}d(u_i)\le10\ell+15.
\end{equation}
We will use (\ref{X_total}) and (\ref{bridge_whole_X})
throughout this paper.
For example,
by (\ref{X_total}) and (\ref{bridge_whole_X})
we have
\begin{equation}
\label{X1_lower}
|X_1|
\ge
\sum_{i=1}^5(2-i)|X_i|+3|X_5|
\ge
10\ell-23+3|X_5|,
\end{equation}
which implies $|X_1|\ge10\ell-23$.
Also, we have
\begin{equation}
\label{X2_estimate} 
|X_2|
\ge
\sum_{i=1}^{5}(3-i)|X_i|-2|X_1|
\ge
20\ell-27-2|X_1|.
\end{equation}

Let $I=\{1,2,3,4,5\}$.
For each pair $i,j\in I$,
we define the set
$X_{ij}=X_i\cap N(u_j)$.
Then we have a partition
\[
X_1=X_{11}\sqcup X_{12}\sqcup X_{13}\sqcup X_{14}\sqcup X_{15}.
\]
We next find that each $X_{1i}$ induces a clique in $G$.
\begin{lem}
\label{X_1i_clique}
For each $i\in I$, the graph $G[X_{1i}]$ is complete.
Consequently, we have
\[
|X_{1i}|\le2\ell-3,\quad
|X_1|\le10\ell-15.
\]
\end{lem}
\begin{proof}
Let $v,v'\in X_{1i}$. 
Since the $6$-set $(U\backslash \{u_i\})\cup \{v,v'\}$ cannot be independent,
it follows that $v$ and $v'$ are adjacent. 
Therefore, $G[X_{1i}\cup \{u_i\}]$ is a complete graph of order
$|X_{1i}|+1\le\omega(G)\le 2\ell-2$. 
In addition, $|X_1|= \sum_{i=1}^{5}|X_{1i}|\le 10\ell-15$.
\end{proof}

Assume, without loss of generality, that
\[
|X_{11}|\ge|X_{12}|\ge|X_{13}|\ge|X_{14}|\ge|X_{15}|
\]
in the rest of our discussion.
Thus (\ref{X1_lower}) implies that
\begin{equation}
\label{X1i_lower}
|X_{15}|=|X_1|-\sum_{i=1}^4|X_{1i}|\ge2\ell-11\ge1.
\end{equation}
In particular, each $X_{1i}$ is nonempty.
The next lemma on the sets $X_{1i}$ is immediate yet useful.
We will use this lemma mainly with the choice $t=2$.

\begin{lem}
\label{flower_lemma}
Let $i,j,k\in I$ be distinct indices and $v\in X_{1i}$.
Suppose that $|X_{1i}|\ge 2\ell-2t$ for some integer with $2\le t\le6$.
Then we have
$|N(v)\cap(X_{1j}\cup X_{1k})|\le2t$ and
$|N(v)\cap X_{1j}|\le2t-1$.
In addition, if $|N(v)\cap X_{1j}|\ge2$ also holds,
then we have $|N(v)\cap X_{1k}|\le2t-3$.
\end{lem}
\begin{proof}
Note that $(\ell-t) K_2\subset G[(X_{1i}\backslash \{v\})\cup \{u_i\}]$. 
Thus, we must avoid $tK_2$ in $G[N(v)\cap (X_{1j}\cup X_{1k})]$
and so the result is immediate as each of $X_{1j}$ and $X_{1k}$ induces a clique.
\end{proof}

\subsection{On the sets $N(u_i)\backslash X_{1i}$}
Besides the same kind of information as Chen and Zhang obtained,
we need some new information on their structural setting.
We start with some observations
on the set $N(u_i)\backslash X_{1i}$.
By Lemma~\ref{degree_lemma} and Lemma~\ref{X_1i_clique}
\[
|N(u_i)\backslash X_{1i}|=d(u_i)-|X_{1i}|\ge 3.
\]

In general,
suppose that $A$ and $B$ are disjoint nonempty sets of vertices in a graph. 
Then it is well-known that,
if $|N(v)\cap B|\ge |A|$ for every $v\in A$,
then there are at least $|A|$ independent edges between $A$ and $B$. 
The following is then immediate.

\begin{lem}
\label{t_matching}
Let $i\in I$ and $t= 2\ell-|X_{1i}|$. 
Every $t$-set $S\subset N(u_i)\backslash X_{1i}$ contains a vertex $v$ for which $|N(v)\cap X_{1i}|< t$. 
\end{lem}
\begin{proof}
If $S= \{v_1,v_2,\ldots,v_t\}\subset N(u_i)\backslash X_{1i}$
and $|N(v_j)\cap X_{1i}|\ge t$ for $1\le j\le t$, then $X_{1i}$ contains $t$ distinct vertices $w_1,w_2,\ldots,w_t$ such that $v_jw_j\in E(G)$ for $1\le j\le t$. 
However then, $(\ell-t)K_2$ in $G[X_{1i}\backslash \{w_1,w_2,\ldots,w_t\}]$ with the $t$ edges $v_jw_j$ $(1\le j\le t)$ forms $\ell K_2$ in $G[N(u_i)]$, a contradiction.
\end{proof}

Let $i\in I$. 
Under certain conditions, there exists a $4$-set $Q_i\subset V(G)$ satisfying 
\begin{equation}
\label{Q_cond}
\text{%
$Q_i\subset N(u_i)\backslash X_{1i}$
and
$Q_i\cup \{w\}$ is an independent $5$-set for every $w\in X_{1i}$.%
} 
\end{equation} 
For example, such $Q_i$ exists if the degree of $u_i$ equals $2\ell+3$,
which we verify next. 

\begin{lem}
\label{Q_lemma}
Let $i\in I$.
If $d(u_i)= 2\ell+3$, then there exists a $4$-set $Q_i$
satisfying \upshape{(\ref{Q_cond})}. 
\end{lem}
\begin{proof}
Note that $|N(u_i)\backslash X_{1i}|= d(u_i)-|X_{1i}|\ge 6$
since $d(u_i)=2\ell+3$.

Suppose first that $|X_{1i}|$ is odd, say $|X_{1i}|= 2t-1$ for some positive integer $t$.
Then $G[N(u_i)\backslash X_{1i}]$ contains $(\ell-t) K_2$
by Lemma~\ref{matching} with four remaining vertices $v_1,v_2,v_3,v_4$. 
Let $Q_i= \{v_1,v_2,v_3,v_4\}$.
Then for any $w\in X_{1i}$,
$G[X_{1i}\backslash \{w\}]$ contains $(t-1) K_2$ as $G[X_{1i}]$ is complete.
Thus $Q_i\cup \{w\}$ must be independent
in order to avoid $\ell K_2$ in $G[N(u_i)]$.

Suppose next that $|X_{1i}|= 2t$ for some positive integer $t$.
Then $G[N(u_i)\backslash X_{1i}]$ contains $(\ell-t-1) K_2$
again by Lemma~\ref{matching} with five remaining vertices $v_1,v_2,\ldots,v_5$. 
Since $G[X_{1i}]$ obviously contains $t K_2$,
the set $Q= \{v_1,v_2,\ldots,v_5\}$ must be independent
to avoid $\ell K_2$ in $G[N(u_i)]$. 
For any $w\in X_{1i}$,
$|N(w)\cap Q|\ge1$ since $Q\cup\{w\}$ cannot be independent.
Also, there cannot be two or more independent edges
between $X_{1i}$ and $Q$ to avoid $\ell K_2$ in $G[N(u_i)]$.
Thus we can find a vertex in $Q$, say $v_5$,
such that every $w\in X_{1i}$ is adjacent to $v_5$.
Let $Q_i=\{v_1,v_2,v_3,v_4\}$.
Recall that $|X_{1i}|=2t\ge2$.
Then $Q_i\cup\{w\}$ must be independent for every $w\in X_{1i}$
to avoid two independent edges between $X_{1i}$ and $Q$.
\end{proof}

\subsection{On the sets $X_2$ and $X_3$}
We next study the sets $X_2$ and $X_3$ in more detail.
We start with the following lemma,
which can be seen as an analogue of Lemma~\ref{X_1i_clique}.
\begin{lem}
\label{bridge_lemma}
Let $i,j,k\in I$ be distinct indices.
\begin{enumerate}
\item[\rm (a)]
If $a\in X_{1i}$, $b\in X_{1j}$, $c\in X_{2i}\cap X_{2j}$, then the set $\{a,b,c\}$ is not independent. 
\item[\rm (b)]
If $a\in X_{1i}$, $b\in X_{1j}$, $c\in X_{1k}$, $d\in X_{3i}\cap X_{3j}\cap X_{3k}$, then the set $\{a,b,c,d\}$ is not independent. 
\end{enumerate}
\end{lem}
\begin{proof}
First, (a) is immediate since the $6$-set $(U\backslash \{u_i,u_j\})\cup \{a,b,c\}$ cannot be independent. 
Similarly, (b) holds by considering the $6$-set $(U\backslash \{u_i,u_j,u_k\})\cup \{a,b,c,d\}$.
\end{proof}

The next lemma is an analogue of Claim 1 of Chen and Zhang~\cite{Chen_Zhang}.
\begin{lem}
\label{intersection_lemma}
Let $i,j,k\in I$ be distinct indices.
\begin{enumerate}
\item[\rm (a)]
If $|X_{1i}|= |X_{1j}|= 2\ell-3$, then $X_{2i}\cap X_{2j}= \emptyset$. 
\item[\rm (b)]
If $|X_{1i}|= |X_{1j}|= |X_{1k}|= 2\ell-3$, then $X_{3i}\cap X_{3j}\cap X_{3k}= \emptyset$. 
\end{enumerate}
\end{lem}
\begin{proof}
Let us begin with (a) by assuming, to the contrary, that $c\in X_{2i}\cap X_{2j}$. 
Since $G[N(c)]$ does not contain $\ell K_2$
while $G[X_{1i}\cup \{u_i\}]\cong G[X_{1j}\cup \{u_j\}]\cong K_{2\ell-2}$,
it follows that $|N(c)\cap (X_{1i}\cup X_{1j})|\le 2\ell-2$. 
Hence, $|(X_{1i}\cup X_{1j})\backslash N(c)|\ge 2\ell-4$. 
Without loss of generality, assume that $|X_{1i}\backslash N(c)|\ge (2\ell-4)/2= \ell-2\ge 4$. 
Note that $X_{1j}\backslash N(c)\ne\emptyset$ as $\omega(G)\le 2\ell-2$. 
Let $b\in X_{1j}\backslash N(c)$. 
By Lemma~\ref{bridge_lemma}~(a) then,
$b$ is adjacent to every vertex in $X_{1i}\backslash N(c)$,
i.e.\ $|N(b)\cap X_{1i}|\ge 4$, contradicting Lemma~\ref{flower_lemma}. 
This proves the assertion (a).

For (b), suppose that $d\in X_{3i}\cap X_{3j}\cap X_{3k}$. 
Then by a similar argument done in (a),
one sees that $|(X_{1i}\cup X_{1j}\cup X_{1k})\backslash N(d)|\ge 4\ell-7$
in order to avoid $\ell K_2$ in $G[N(d)]$. 
Without loss of generality, suppose that
$|X_{1i}\backslash N(d)|\ge |X_{1j}\backslash N(d)|\ge |X_{1k}\backslash N(d)|$. 
Then
\[
6
\le
\lceil{(4\ell-7)/3}\rceil
\le
|X_{1i}\backslash N(d)|
\le
|X_{1i}|\le 2\ell-3
\]
and
\[
|X_{1j}\backslash N(d)|
\ge
(4\ell-7-|X_{1i}|)/2
\ge
\ell-2
\ge
4.
\]
Also, $X_{1k}\backslash N(d)\ne\emptyset$ in order to avoid $K_{2\ell-1}$ in $G$.
Let
\[
\{a_1,a_2,\ldots,a_6\}\subset X_{1i}\backslash N(d),\ \ 
\{b_1,b_2,b_3,b_4\}\subset X_{1j}\backslash N(d)\ \ 
\text{and}\ \ 
c\in X_{1k}\backslash N(d).
\]
By Lemma~\ref{flower_lemma},
none of $|N(c)\cap X_{1i}|,|N(c)\cap X_{1j}|,|N(b_1)\cap X_{1i}|$ exceeds $3$. 
By this fact with Lemma~\ref{bridge_lemma}~(b),
we may assume that $b_1c\notin E(G)$, and $N(c)\cap X_{1i}=\{a_1,a_2,a_3\}$.
Then again by Lemma~\ref{flower_lemma},
at most one of $b_2,b_3,b_4$ is adjacent to $c$
and so suppose that $b_2c\notin E(G)$. 
Thus, by Lemma~\ref{bridge_lemma},
$N(b_1)\cap X_{1i}= N(b_2)\cap X_{1i}= \{a_4,a_5,a_6\}$.
This gives us $(\ell-2) K_2$ in $G[(X_{1j}\backslash \{b_1,b_2\})\cup \{u_j\}]$
with two edges $a_4a_5$ and $a_6b_2$ in $G[N(b_1)]$,
which is impossible.
\end{proof}

We let
\[
J:=\Set{i\in I|
\text{a subset $Q_i\subset N(u_i)\backslash X_{1i}$ satisfying (\ref{Q_cond}) exists}
}.
\]
The next lemma describes how the sets $Q_i$ intersect each other.

\begin{lem}
\label{Q_intersect}\mbox{}
\begin{enumerate}
\item[\rm (a)]
Let $i,j\in J$ be indices with $i<j$. 
If $|X_{1i}|+|X_{1j}|\ge 2\ell-1$, then we have $Q_i\cap Q_j\cap X_2= \emptyset$. 
Consequently, if $|X_{14}|+|X_{15}|\ge 2\ell-1$, then
\[\sum_{i\in J}|Q_i\cap X_2|\le |X_2|.\]
\item[\rm (b)]
Let $i,j,k\in J$ be indices with $i<j<k$. 
If $|X_{1i}|+|X_{1j}|+|X_{1k}|\ge 6\ell-14$,
then $Q_i\cap Q_j\cap Q_k\cap X_3= \emptyset$. 
Consequently,
if $|X_{13}|+|X_{14}|+|X_{15}|\ge 6\ell-14$,
then
\[\sum_{i\in J}|Q_i\cap X_3|\le 2|X_3|. \]
\end{enumerate}
\end{lem}
\begin{proof}
For (a), suppose that $c\in Q_i\cap Q_j\cap X_2$. 
Then no vertex in $X_{1i}\cup X_{1j}$ is adjacent to $c$ by (\ref{Q_cond}),
so $G[X_{1i}\cup X_{1j}]$ is a clique by Lemma~\ref{bridge_lemma} (a). 
However then, $\omega(G)\ge |X_{1i}|+|X_{1j}|\ge 2\ell-1$. 
As a result, $Q_i\cap Q_j\cap X_2= \emptyset$. 

For (b), suppose that $d\in Q_i\cap Q_j\cap Q_k\cap X_3$.
Then no vertex in $X_{1i}\cup X_{1j}\cup X_{1k}$ is adjacent to $d$
by (\ref{Q_cond}).
By the assumption $|X_{1i}|+|X_{1j}|+|X_{1k}|\ge 6\ell-14$,
we have
\[
2\ell-4\le|X_{1i}|\le2\ell-3,\quad
2\ell-5\le|X_{1j}|\le2\ell-3,\quad
2\ell-8\le|X_{1k}|\le2\ell-3.
\]
We consider two cases separately according to whether $|X_{1k}|\le2\ell-7$ or not.

We first consider the case $|X_{1k}|\le2\ell-7$.
In this case, we have $|X_{1j}|\ge2\ell-4$ by the assumption.
Therefore we have $|X_{1i}|,|X_{1j}|\ge8$.
Take a vertex $c\in X_{1k}$.
By Lemma~\ref{flower_lemma} with $t=4$,
we find that $|N(c)\cap X_{1i}|, |N(c)\cap X_{1j}|\le 7$.
Thus we can take vertices
$a\in X_{1i}\setminus N(c)$ and $b\in X_{1j}\setminus N(c)$.
By Lemma~\ref{flower_lemma} with $t=2$,
we find that $|N(a)\cap X_{1j}|, |N(b)\cap X_{1i}|\le 3$,
i.e. $|X_{1j}\setminus N(a)|, |X_{1i}\setminus N(b)|\ge 5$.
By Lemma~\ref{bridge_lemma}~(b),
this implies $|N(c)\cap(X_{1i}\cup X_{1j})|\ge10$,
contradicting Lemma \ref{flower_lemma} with $t=4$.

We next consider the case $|X_{1k}|\ge2\ell-6$.
In this case, we have $|X_{1i}|\ge8$, $|X_{1j}|\ge7$ and $|X_{1k}|\ge6$.
Thus we can take vertices
\[
\{a_1,a_2,\ldots,a_8\}\subset X_{1i},\quad
\{b_1,b_2,\ldots,b_7\}\subset X_{1j}\quad\text{and}\quad
c\in X_{1k}.
\]
By Lemma~\ref{flower_lemma} with $t=3$,
we find that $|N(c)\cap X_{1i}|, |N(c)\cap X_{1j}|, |N(b_1)\cap X_{1i}|\le 5$.
Thus we may assume that $a_1c, b_1c, a_2b_1, a_3b_1\not\in E(G)$.
By Lemma~\ref{bridge_lemma}~(b),
we find that $a_1\in N(b_1)\cap X_{1i}$ and $a_2,a_3\in N(c)\cap X_{1i}$.
Also, by Lemma~\ref{flower_lemma} with $t=2$,
we see that $|N(a_1)\cap X_{1j}|\le3$.
Thus we may assume that $a_1b_2, a_1b_3, a_1b_4, a_1b_5\not\in E(G)$.
However, again by Lemma~\ref{bridge_lemma}~(b),
we find that $b_2,b_3,b_4,b_5\in N(c)\cap X_{1j}$,
contradicting Lemma \ref{flower_lemma} with $t=3$.
\end{proof}

\subsection{On the set $X_1$}
Based on the above observations,
we next give a closer look at $X_1$.
Recall that $10\ell-23\le|X_1|\le10\ell-15$.
The next lemma gives an improved upper bound for $|X_1|$.

\begin{lem}
\label{X14_reduction}
We have $|X_{14}|\le 2\ell-4$. 
Consequently, $|X_1|\le 10\ell-17$. 
\end{lem}
\begin{proof}
If the statement is false, then $|X_{1i}|= 2\ell-3$ for $1\le i\le 4$.
Thus, $X_2,X_3\subset N(u_5)$ by Lemma~\ref{intersection_lemma}. 
Then $X_{15}\cup X_2\cup X_3\cup X_5\subset N(u_5)$
and so $|X_{15}|+|X_2|+|X_3|+|X_5|\le d(u_5)\le 2\ell+3$. 
This then implies that 
\[
10\ell-4
=\sum_{i=1}^{4}|X_{1i}|+(|X_{15}|+|X_2|+|X_3|+|X_5|)+|X_4|
\le10\ell-9+|X_4|,
\]
so that $|X_4|\ge5$. 
On the other hand, using (\ref{X_total}) and (\ref{bridge_whole_X}), 
\begin{align*} 
2|X_4|\le\sum_{i=2}^{5}(i-2)|X_i|\le8, 
\end{align*} 
so that $|X_4|\le 4$. 
Since this is impossible, it follows that $|X_{14}|\le 2\ell-4$ and so
$|X_1|\le 3(2\ell-3)+2(2\ell-4)= 10\ell-17$.
This completes the proof.
\end{proof}

\begin{lem}
\label{Onozuka}
The graph $G$ contains a $5$-set of independent vertices
\[U= \{u_1,u_2,\ldots,u_5\}\subset V(G)\]
for which either
{\rm (A)} $|X_1|\in \{ 10\ell-18,10\ell-17 \}$ or 
{\rm (B)} $X_5\ne \emptyset$,
i.e.\ each of the five vertices in $U$ is adjacent to a common vertex.
\end{lem}
\begin{proof}
Let $U= \{u_1,u_2,\ldots,u_5\}$ be a $5$-set of independent vertices in $G$
and suppose that $U$ does not satisfy (A).
By Lemma~\ref{X14_reduction},
we have $|X_1|\le10\ell-19$.
Then by (\ref{X_total}) and (\ref{bridge_whole_X}), we have
\[
\sum_{i=1}^5d(u_i)
=
\sum_{i=1}^5i|X_i|
\ge
2\sum_{i=1}^5|X_i|-|X_1|
\ge
10\ell+11
\]
so that $d(u_i)= 2\ell+3$ for some $i\in I$. 
By Lemma~\ref{Q_lemma}, therefore, we may assume the existence of a $4$-set $Q_i$ that satisfies (\ref{Q_cond}). 
Select a vertex $v\in X_{1i}$ and consider the $5$-set $U'= Q_i\cup \{v\}$, which must be independent. 
Note also that each of the five vertices in $U'$ is adjacent to $u_i$. 
Thus, $U'$ satisfies (B).
\end{proof}

For the rest of this paper,
$U$ denotes a fixed 5-set of independent vertices in $G$
satisfying either (A) or (B) in Lemma~\ref{Onozuka}.

\subsection{Additional lemmas when $|X_{13}|=2\ell-3$}
\mbox{}

We prepare some more lemmas on $X_2$ under the condition $|X_{13}|=2\ell-3$.
\begin{lem}
\label{X2_45}
Suppose that $|X_{13}|= 2\ell-3$. 
Then $|X_2|= |X_{24}|+|X_{25}|-|X_{24}\cap X_{25}|$. 
Furthermore, $|X_{24}\cap X_{25}|\le 1$. 
\end{lem}
\begin{proof}
The first assertion is immediate by Lemma~\ref{intersection_lemma}~(a). 
By Lemma~\ref{intersection_lemma}~(a)~(b),
we see that $X_2,X_3\subset N(u_4)\cup N(u_5)$. 
Also, obviously $X_4,X_5\subset N(u_4)\cup N(u_5)$. 
Thus
\begin{align*} 
4\ell+6
\ge |N(u_4)|+|N(u_5)|
&\ge |X_{14}|+|X_{15}|+|X_{24}|+|X_{25}|+|X_3|+|X_4|+|X_5|\\
&=|X_{24}\cap X_{25}|+\sum_{i=1}^{5}|X_i|-(|X_{11}|+|X_{12}|+|X_{13}|)\\
&=|X_{24}\cap X_{25}|+4\ell+5,
\end{align*}
so that $|X_{24}\cap X_{25}|\le 1$.
This proves the lemma.
\end{proof}

\begin{lem}
\label{not_in_45} 
Suppose that $|X_{13}|=2\ell-3$ and $v\in X_{2i}$ for some $i\in\{4,5\}$.
\begin{enumerate}
\item[\rm (a)] 
If $v\notin X_{24}\cap X_{25}$, then $|N(v)\cap X_{1i}|\ge |X_{1i}|-3$. 
\item[\rm (b)] 
If $v\notin X_{24}\cap X_{25}$ and $|X_{1i}|= 2\ell-4$,
then $G[X_{1i}\cup\{v\}]\cong K_{2\ell-3}$. 
\end{enumerate}
\end{lem}
\begin{proof}
Suppose that $v\notin X_{24}\cap X_{25}$,
that is, $v\in X_{2j}$ for some $j\in\{1,2,3\}$. 
Since $\omega(G)\le 2\ell-2$, there is $w\in X_{1j}\backslash N(v)$. 
By Lemma~\ref{bridge_lemma}~(a),
observe that $w$ is adjacent to each vertex in $X_{1i}\backslash N(v)$. 
Thus, $|X_{1i}\backslash N(v)|\le|N(w)\cap X_{1i}|\le3$
by Lemma~\ref{flower_lemma}. 
This proves (a). 

We next prove the assertion (b).
Assume, to the contrary, that $v\notin X_{24}\cap X_{25}$, $|X_{1i}|=2\ell-4$
and $G[X_{1i}\cup\{v\}]$ is not complete.
Suppose $v\in X_{2j}$ with $j\in\{1,2,3\}$.
By the same argument above, we see that
$|X_{1i}\backslash N(v)|,|X_{1j}\backslash N(v)|\le 3$
since we can take some vertex $w'\in X_{1i}\setminus N(v)$ by the assumption.
However then, $|N(v)\cap (X_{1i}\cup X_{1j})|\ge |X_{1i}|+|X_{1j}|-6= 4\ell-13\ge 2\ell-1$, 
producing $\ell K_2$ in $G[N(v)\cap(X_{1i}\cup X_{1j}\cup \{u_i,u_j\})]$. 
Since this does not occur, (b) also holds.
\end{proof}

\begin{lem}
\label{45_estimate}
Suppose that $|X_{13}|= 2\ell-3$.
If $|X_{1i}|\ge 2\ell-5$ for some $i\in\{4,5\}$,
then $|X_{1i}|+|X_{2i}|\le 2\ell$.
\end{lem}
\begin{proof}
Note that $|X_{1i}|\in\{2\ell-5,2\ell-4\}$ by Lemma~\ref{X14_reduction}.

If $|X_{1i}|= 2\ell-4$ and $|X_{2i}|\ge 5$, then $X_{2i}$ contains four vertices $v_1,v_2,v_3,v_4$ such that $|N(v_j)\cap X_{1i}|\ge |X_{1i}|-3\ge 5$ for $1\le j\le 4$ by Lemma~\ref{not_in_45}~(a) since $|X_{24}\cap X_{25}|\le 1$. 
This however contradicts Lemma~\ref{t_matching}. 
Thus, $|X_{2i}|\le 4$, as desired. 

Similarly, if $|X_{1i}|= 2\ell-5$ and $|X_{2i}|\ge 6$, say $\{v_1,v_2,\ldots,v_6\}\subset X_{2i}$, then we may assume that $v_1v_2\in E(G)$ as $\alpha(G)= 5$ and $|N(v_j)\cap X_{1i}|\ge |X_{1i}|-3\ge 4$ for $3\le j\le 5$. 
Thus, there are three distinct vertices $w_3,w_4,w_5\in X_{1i}$ such that $v_jw_j\in E(G)$ for $3\le j\le 5$. 
These three edges with the edge $v_1v_2$ and $(\ell-4) K_2$ in $G[X_{1i}\backslash \{w_3,w_4,w_5\}]$ result in $\ell K_2$ in $G[N(u_i)]$, again a contradiction. 
It follows that $|X_{2i}|\le 5$.
\end{proof}

\section{Completion of the proof}
We are now prepared to prove Theorem~\ref{upper_bound}.
Recall that we are now considering a fixed 5-set $U$ of independent vertices in $G$
satisfying one of the conditions
\begin{center}
{\rm (A)} $|X_1|\in \{ 10\ell-18,10\ell-17 \}$\quad
or\quad 
{\rm (B)} $X_5\ne \emptyset$.
\end{center}
In the following subsections,
we consider these two cases separately.

\subsection{Case A: $|X_1|\in\{10\ell-18,10\ell-17\}$}
\mbox{}

In this subsection, we consider Case A\@.
In this case, observe that $|X_{11}|=|X_{12}|=2\ell-3$ and $|X_{14}|=2\ell-4$
since otherwise we have $|X_1|\le10\ell-19$ by Lemma~\ref{X14_reduction}.
Thus we also have $|X_{13}|+|X_{15}|\in\{4\ell-8,4\ell-7\}$.
Recalling (\ref{X2_estimate}), we have
\[
|X_2|
\ge20\ell-27-2|X_1|
=\left\{
\begin{array}{ll}
9 &(\text{if $|X_{13}|+|X_{15}|= 4\ell-8$}),\\
7 &(\text{if $|X_{13}|+|X_{15}|= 4\ell-7$}).
\end{array}
\right.
\]
We now consider two subcases, according to the value of $|X_{13}|$.

\subsubsection{}
\textbf{Subcase A1}: $|X_{13}|=2\ell-3$.\\
\indent
In this case, $|X_{15}|\in \{2\ell-5,2\ell-4\}$.
Again by (\ref{X2_estimate}) with Lemmas~\ref{X2_45} and \ref{45_estimate},
\begin{equation}
\label{X2_A1_proof}
\begin{aligned}
|X_2|
&\ge20\ell-27-2|X_1|\\
&=8\ell-9-2(|X_{14}|+|X_{15}|)\\
&\ge2(|X_{24}|+|X_{25}|)-9\\
&= 2(|X_2|+|X_{24}\cap X_{25}|)-9.
\end{aligned}
\end{equation}
Thus, 
\begin{equation}
\label{X2_A1}
9-2|X_{24}\cap X_{25}|
\ge|X_2| 
\ge\left\{
\begin{array}{ll}
9 &(\text{if $|X_{15}|= 2\ell-5$}),\\
7 &(\text{if $|X_{15}|= 2\ell-4$}).
\end{array}
\right.
\end{equation}

First, if $X_{24}\cap X_{25}\ne \emptyset$,
then it is immediate by (\ref{X2_A1_proof}) and (\ref{X2_A1})
that $|X_{1i}|= 2\ell-|X_{2i}|= 2\ell-4\ge 8$ for $i= 4,5$.
Let $X_{24}\cap X_{25}= \{v\}$.
By Lemma~\ref{not_in_45}~(b),
each of the three vertices in $X_{2i}\backslash \{v\}$
is adjacent to every vertex in $X_{1i}$ for $i= 4,5$.
Since neither $G[N(u_4)]$ nor $G[N(u_5)]$ contains $\ell K_2$,
we find that $N(v)\cap (X_{14}\cup X_{15})= \emptyset$.
However then, $G[X_{14}\cup X_{15}]\cong K_{4\ell-8}$
by Lemma~\ref{bridge_lemma}~(a),
which cannot occur since $4\ell-8>2\ell-2\ge \omega(G)$.

Thus, we may assume that $X_{24}\cap X_{25}= \emptyset$.
By (\ref{X2_A1_proof}) and (\ref{X2_A1}) again,
$|X_{1i}|= 2\ell-|X_{2i}|= 2\ell-4$ for some $i\in \{4,5\}$. 
Then Lemma~\ref{not_in_45}~(b) implies
that each of the four vertices in $X_{2i}$ is adjacent to every vertex in $X_{1i}$.
Thus, $|N(v)\cap X_{1i}|= |X_{1i}|= 2\ell-4\ge 8$ for every $v\in X_{2i}$.
However, this contradicts Lemma~\ref{t_matching}.

\subsubsection{}
\textbf{Subcase A2}: $|X_{13}|=2\ell-4$.\\
\indent
In this case, $|X_{1i}|= 2\ell-4$ for $3\le i\le 5$ and $|X_1|= 10\ell-18$. 

We first verify that,
if $v\in X_2$, then $v\in X_{2j}$ and $|N(v)\cap X_{1j}|\ge 5$
for some $j\in \{3,4,5\}$. 
To see this, suppose that $v\in X_{2i}\cap X_{2j}$ with $1\le i< j\le 5$. 
If $i\le 2$, then $3\le j\le 5$ as $X_{21}\cap X_{22}=\emptyset$
by Lemma~\ref{intersection_lemma}. 
Since $\omega(G)\le 2\ell-2$ and $|X_{1i}|= 2\ell-3$,
there is a vertex $w\in X_{1i}$ such that $vw\notin E(G)$. 
Then $w$ is adjacent to every vertex in $X_{1j}\backslash N(v)$
by Lemma~\ref{bridge_lemma}~(a),
implying that $|X_{1j}\backslash N(v)|\le 3$ by Lemma~\ref{flower_lemma}. 
If $3\le i< j\le 5$, on the other hand,
then a similar reasoning shows
that $|X_{1i}\backslash N(v)|\le 3$ or $|X_{1j}\backslash N(v)|\le 3$.
Consequently, $|X_{1j}\backslash N(v)|\le 3$ for some $j\in \{3,4,5\}$ in each case,
i.e.\ $|N(v)\cap X_{1j}|\ge |X_{1j}|-3\ge 5$. 

We next verify that $|X_2|= 9$. 
If this is not the case, then $|X_2|\ge 10$. 
By the pigeonhole principle,
there exists an integer $j\in \{3,4,5\}$ and a $4$-set $S\subset X_{2j}$
such that $|N(v)\cap X_{1j}|\ge 5$ for each $v\in S$,
contradicting Lemma~\ref{t_matching}. 

Recalling (\ref{X2_estimate}), we have 
\[
9= |X_2|\ge |X_2|-|X_4|-2|X_5|=\sum_{i=1}^{5}(3-i)|X_i|-2|X_1|\ge 9
\]
and so $|X_4|= |X_5|= 0$, which further tells us that $|X_3|= 5$. 
Also, $\sum_{i=1}i|X_{i}|= 10\ell+15$, that is, $d(u_i)= 2\ell+3$ for $1\le i\le 5$. 
By Lemma~\ref{Q_lemma},
we can find five $4$-sets $Q_1,Q_2,\ldots,Q_5$ satisfying (\ref{Q_cond}). 
By Lemma~\ref{Q_intersect}~(a)~(b) then, 
\begin{equation}
\label{Q_summation} 
20
=\sum_{i=1}^{5}|Q_i| 
=\sum_{i=1}^{5}|Q_i\cap X_2|+\sum_{i=1}^{5}|Q_i\cap X_3| 
\le|X_2|+2|X_3|= 19, 
\end{equation}
which is clearly impossible.

As a result, Subcases A1 and A2 are both impossible, i.e.\ Case A never occurs.

\subsection{Case B: $X_5\neq\emptyset$}
\mbox{}

In this subsection, we next consider Case B\@.
In this case,
we see that $|X_1|\ge 10\ell-23+3|X_5|\ge 10\ell-20$ by (\ref{X1_lower}).
Since Case A never occurs, we may assume that $|X_1|\in \{10\ell-20,10\ell-19\}$.
We again consider two subcases.

\subsubsection{}
\textbf{Subcase B1}: $|X_1|=10\ell-20$.\\
\indent
By (\ref{X1_lower}),
we find that $|X_3|=| X_4|=0$ and $|X_5|=1$, which in turn implies that $|X_2|= 15$.
Also, $\sum_{i=1}^{5}i|X_i|= 10\ell+15$
and so $d(u_i)= 2\ell+3$ for $1\le i\le 5$,
which guarantees the existence of five $4$-sets $Q_1,Q_2,\ldots,Q_5$
satisfying (\ref{Q_cond}). 
As done in (\ref{Q_summation}), we can write 
\[
20
=\sum_{i=1}^{5}|Q_i| 
=\sum_{i=1}^{5}|Q_i\cap X_2|+\sum_{i=1}^{5}|Q_i\cap X_5|
\le
|X_2|+5=20
\]
and so $\sum_{i=1}^{5}|Q_i\cap X_2|= 15$ and $\sum_{i=1}^{5}|Q_i\cap X_5|= 5$.
Hence, if we let $X_5= \{v_0\}$,
then each $Q_i$ consists of three vertices in $X_2$ and $v_0$
so that $\bigcup_{i=1}^{5}Q_i= X_2\cup X_5$.
However then, no vertex in $X_1\cup X_2$ is adjacent to $v_0$ as $v_0\in Q_i$,
implying that $N(v_0)= U$. Thus, $d(v_0)= 5< 2\ell=\delta(G)$, a contradiction. 

\subsubsection{}
\textbf{Subcase B2}: $|X_1|=10\ell-19$.\\
\indent
First, observe that
\[
3
\le
|X_3|+2|X_4|+3|X_5|
=
\sum_{i=1}^{5}(i-2)|X_i|+|X_1| \le 4
\]
by (\ref{X_total}) and (\ref{bridge_whole_X}).
Thus, we find that $|X_4|=0$, $|X_5|=1$ and either
\begin{center}
(i) $|X_2|= 13$ and $|X_3|= 1$\quad
or\quad
(ii) $|X_2|= 14$ and $|X_3|= 0$. 
\end{center}
We consider these two cases separately.
Let $X_5= \{v_0\}$. 

If (i) occurs,
then again $\sum_{i=1}^{5}i|X_i|=10\ell+15$
and so we have $4$-sets $Q_1,Q_2,\ldots,Q_5$ satisfying (\ref{Q_cond}).
Since $|X_1|=10\ell-19$,
we can check that $|X_{13}|+|X_{14}|+|X_{15}|\ge6\ell-13$.
Thus by Lemma~\ref{Q_intersect}~(a)~(b),
\begin{align*}
20
&
=\sum_{i=1}^{5}|Q_i|
=\sum_{i=1}^{5}|Q_i\cap X_2|
+\sum_{i=1}^{5}|Q_i\cap X_3|
+\sum_{i=1}^{5}|Q_i\cap X_5|\\
&\le|X_2|+2|X_3|+5=20
\end{align*} 
and we arrive at the same contradiction as in Subcase~B.1.

Finally, suppose that (ii) occurs.
Then $\sum_{i=1}^{5}i|X_i|=10\ell+14$.
Thus, four of the five vertices in $U$ have degree~$2\ell+3$
and the remaining one has degree~$2\ell+2$.
Let $i_0\in I$ for which $d(u_{i_0})=2\ell+2$. 
We know that there is a $4$-set $Q_i$
satisfying (\ref{Q_cond}) for each $i\in I\backslash \{i_0\}$.

In the set $N(u_{i_0})\backslash X_{1i_0}$,
which contains $2\ell+2-|X_{1{i_0}}|\ge 5$ vertices,
we show that there exists a $3$-set $Q\subset N(u_{i_0})\backslash X_{1i_0}$
such that $|N(v)\cap X_{1{i_0}}|\le 2$ for each $v\in Q$.
Recall Lemma~\ref{matching}.
First, if $|X_{1{i_0}}|=2t-1$ for some positive integer $t$,
then $G[N(u_{i_0})\backslash X_{1i_0}]$ contains $(\ell-t-1) K_2$
and five vertices $v_1,v_2,\ldots,v_5$.
Since $G[N(u_{i_0})]$ cannot contain $\ell K_2$,
there are at most two independent edges
between $X_{1{i_0}}$ and $\{v_1,v_2,\ldots,v_5\}$.
Thus, we may assume that $|N(v_j)\cap X_{1{i_0}}|\le 2$ for $1\le j\le 3$.
Similarly, if $|X_{1{i_0}}|= 2t$ for some positive integer $t$,
then $G[N(u_{i_0})\backslash X_{1i_0}]$ contains $(\ell-t-1) K_2$
and four vertices $v_1,v_2,v_3,v_4$. 
Then there cannot be two or more independent edges
between $X_{1{i_0}}$ and $\{v_1,v_2,v_3,v_4\}$. 
Hence, we may assume that $|N(v_j)\cap X_{1{i_0}}|\le 1$ for $1\le j\le 3$. 
As a result,
we find a $3$-set $Q= \{v_1,v_2,v_3\}\subset N(u_{i_0})\backslash X_{1i_0}$
such that $|N(v)\cap X_{1{i_0}}|\le 2$ for each $v\in Q$.

Write $Q'_{i_0}= Q$ and $Q'_i=Q_i$ for each $i\in I\backslash \{i_0\}$.
We next verify that $\sum_{i=1}^{5}|Q'_i\cap X_2|\le |X_2|$
by proving that $Q'_i\cap Q'_j\cap X_2= \emptyset$ for distinct integers $i,j\in I$.
First, since $|X_1|= 10\ell-19$, it follows that
$|X_{14}|+|X_{15}|=|X_1|-\sum_{i=1}^{3}|X_{1i}|\ge 4\ell-10\ge2\ell+2$.
Hence, by Lemma~\ref{Q_intersect}~(a),
it suffices to verify the result for $i\in I\backslash {i_0}$ and $j= i_0$.
If $v\in Q'_i\cap Q'_j\cap X_2$, then no vertex in $X_{1i}$ is adjacent to $v$.
Also, we have just shown that $v$ is adjacent to at most two vertices in $X_{1j}$.
Then by Lemma~\ref{bridge_lemma}~(a),
there are at least $|X_{1j}|-2$ vertices in $X_{1j}$
that are adjacent to every vertex in $X_{1i}$.
Hence,
$\omega(G)\ge|X_{1i}|+|X_{1j}|-2\ge|X_{14}|+|X_{15}|-2\ge 2\ell$,
which cannot occur.

Therefore,
\[
19
=\sum_{i=1}^{5}|Q'_i|
=\sum_{i=1}^{5}|Q'_i\cap X_2|+\sum_{i=1}^{5}|Q'_i\cap X_5|
\le|X_2|+5= 19,
\] 
so $v_0$ belongs to every $Q'_i$.
Thus, $N[v_0]\subset U\cup Q'_{i_0}\cup X_{1{i_0}}$.
In particular, $|N(v_0)\cap X_{1{i_0}}|\le 2$
and so $d(v_0)\le |U|+|Q'_{i_0}\backslash \{v_0\}|+2= 9< 2\ell$.
This is again impossible.

We conclude that Case B never occurs, either.

This completes
the proof of Theorem~\ref{upper_bound} and Theorem~\ref{main_thm}.

\subsection*{Acknowledgements}
The authors would like to thank Prof. Futaba Fujie
for her comments and suggestions,
which greatly improved the original manuscript.

\vspace{1mm}
\begin{flushleft}
{\footnotesize
{\sc
Graduate School of Mathematics, Nagoya University,\\
Chikusa-ku, Nagoya 464-8602, Japan.
}\\
{\it E-mail address}, S. Kadota: {\tt m13018c@math.nagoya-u.ac.jp}\\
{\it E-mail address}, Y. Suzuki: {\tt m14021y@math.nagoya-u.ac.jp}\\[2mm]
{\sc
Toyota Technological Institute,\\
Hisakata, Tenpaku-ku, Nagoya 468-8511, Japan.
}\\
{\it E-mail address}, T. Onozuka: {\tt onozuka@toyota-ti.ac.jp}\\
}
\end{flushleft}
\end{document}